\documentclass[english]{llncs}

\makeatletter

\newcommand{\lyxmathsym}[1]{\ifmmode\begingroup\def\b@ld{bold}
  \text{\ifx\math@version\b@ld\bfseries\fi#1}\endgroup\else#1\fi}
\newcommand{\qedhere}{\tag*{$\square$}}

\usepackage{lmodern}
\usepackage{xcolor}
\usepackage{amsmath}
\usepackage{amssymb, xfrac, hyperref}

\def\LP{\operatorname{LP}}
\def\IP{\operatorname{IP}}

\makeatother

\begin{document}
\title{Improving the Cook et al. Proximity Bound Given Integral Valued Constraints}
\author{Marcel Celaya\inst{1} \and Stefan Kuhlmann\inst{2} \and Joseph
Paat\inst{3} \and Robert Weismantel\inst{1}}
\institute{Department of Mathematics, Institute for Operations Research, ETH
Z\"{u}rich, Switzerland\\
\email{\{marcel.celaya, robert.weismantel\}@ifor.math.ethz.ch}\\
\and Institut f\"{u}r Mathematik, Technische Universit\"{a}t Berlin,
Germany\\
\email{kuhlmann@math.tu-berlin.de}\\
\and Sauder School of Business, University of British Columbia, Vancouver
BC, Canada\\
\email{joseph.paat@sauder.ubc.ca}}
\maketitle
\thispagestyle{plain}

\begin{abstract}
\global\long\def\ve#1{\boldsymbol{#1}}%
\global\long\def\R{\mathbb{R}}%
\global\long\def\Z{\mathbb{Z}}%
\global\long\def\aa{\ve a}%
\global\long\def\AA{\ve A}%
\global\long\def\BB{\ve B}%
\global\long\def\dd{\ve d}%
\global\long\def\MM{\ve M}%
\global\long\def\TT{\ve T}%
\global\long\def\UU{\ve U}%
\global\long\def\bb{\ve b}%
\global\long\def\cc{\ve c}%
\global\long\def\rr{\ve r}%
\global\long\def\dd{\ve d}%
\global\long\def\ee{\mathbf{e}}%
\global\long\def\qq{\ve q}%
\global\long\def\uu{\ve u}%
\global\long\def\vv{\ve v}%
\global\long\def\ww{\ve w}%
\global\long\def\xx{\ve x}%
\global\long\def\dede{\ve{\delta}}%
\global\long\def\tt{\ve t}%
\global\long\def\Q{\mathbb{Q}}%
\global\long\def\yy{\ve y}%
\global\long\def\zz{\ve z}%
\global\long\def\aalpha{\ve{\alpha}}%
\global\long\def\llambda{\ve{\lambda}}%
\global\long\def\zero{\mathbf{0}}%
\global\long\def\one{\mathbf{1}}%
\global\long\def\supp{\mathrm{supp}}%
\global\long\def\vol{\mathrm{vol}}%
\global\long\def\cone{\mathrm{cone}}%
\global\long\def\sign{\mathrm{sign}}%
\global\long\def\conv{\mathrm{conv}}%
\global\long\def\ccone{\mathcal{C}}%
\global\long\def\poly{\mathcal{P}}%
\global\long\def\norm#1{\left\Vert #1\right\Vert }%
\global\long\def\dist{\mathrm{dist}}%
\global\long\def\t{\top}%
\global\long\def\spn{\mathrm{span}}%
\global\long\def\cubes{\Sigma}%
\global\long\def\image{\mathrm{im}}%
\global\long\def\feascone{\mathcal{F}}%
\global\long\def\tightcone{\mathcal{T}}%
\global\long\def\nonneg{\mathcal{N}}%
\global\long\def\spindle{\mathcal{S}}%
\global\long\def\Gram{\text{\ensuremath{\mathrm{Gram}}}}%
\global\long\def\arrangement{\mathcal{A}}%
\global\long\def\basis{\mathcal{B}}%
\global\long\def\qoly{\mathcal{Q}}%
\global\long\def\itr{\text{int}}%
\global\long\def\prox{\kappa}%
\global\long\def\objfn{\aalpha}%

\global\long\def\AAhat{\hat{\AA\;}\negthickspace}%
Consider a linear program of the form $\max\;\cc^{\t}\xx:\AA\xx\leq\bb$, where $\AA$ is an $m\times n$ integral matrix. 
In 1986 Cook, Gerards, Schrijver, and Tardos proved that, given an optimal solution $\xx^{*}$, if an optimal integral solution $\zz^{*}$ exists, then it may be chosen such that $\left\Vert \xx^{*}-\zz^{*}\right\Vert _{\infty}<n\Delta$, where $\Delta$ is the largest magnitude of any subdeterminant of $\AA$. 
Since then an open question has been to improve this bound, assuming that $\bb$ is integral valued too. 
In this manuscript we show that $n\Delta$ can be replaced with $\sfrac{n}{2}\cdot\Delta$ whenever $n\geq2$. 
We also show that, in certain circumstances, the factor $n$ can be removed entirely.
\end{abstract}

\section{Introduction.}

Suppose $\AA$ is an integral full-column-rank $m\times n$ matrix.
The polyhedron corresponding to a right hand side $\bb\in\Q^{m}$ is 
\begin{align*}
\poly\left(\AA,\bb\right) & :=\left\{ \xx\in\R^{n}:\ \AA\xx\leq\bb\right\} .\\
\intertext{\text{The linear program corresponding to \ensuremath{\poly(\AA,\bb)} and an objective vector \ensuremath{\cc\in\Q^{n}} is}}\LP(\AA,\bb,\cc) & :=\max\left\{ \cc^{\top}\xx:\ \xx\in\poly(\AA,\bb)\right\} ,\\
\intertext{\text{and the corresponding integer linear program is}}\IP(\AA,\bb,\cc) & :=\max\left\{ \cc^{\top}\xx:\ \xx\in\poly(\AA,\bb)\cap\Z^{n}\right\} .
\end{align*}
The \textit{proximity} question in integer linear programming can
be stated as follows: 
Given an optimal vertex solution $\xx^{*}$ of $\LP(\AA,\bb,\cc)$, how far away is the nearest optimal solution $\zz^{*}$ to $\IP(\AA,\bb,\cc)$ (if one exists)? 
Proximity has a wide array of applications in integer linear programming. 
Perhaps not too surprisingly, upper bounds on proximity can help identify integer vectors in $\poly(\AA,\bb)$ from vertices; this is relevant in search techniques such as the feasibility pump~\cite{FGL2005} and dynamic programming~\cite{EW2018,JR2018}. 
An upper bound of $\pi$ on proximity also leads to a trivial enumeration algorithm to optimize $\IP(\AA,\bb,\cc)$: Solve $\LP(\AA,\bb,\cc)$ to identify an optimal vertex $\xx^{*}$ and then enumerate all $O(\pi^{n})$ many integer points around $\xx^{*}$. 
This simple enumeration algorithm shows that $\IP(\AA,\bb,\cc)$ can be solved in $O(2^{O(n)})$ time, provided one can establish a value of $\pi$ independent of $n$. 

Proximity has been studied for decades with perhaps the most foundational result due to Cook, Gerards, Schrijver, and Tardos. 
To state their result, we denote the largest absolute $k\times k$ minor of $\AA$ by 
\[
\Delta_{k}\left(\AA\right):=\max\left\{ \left|\det\MM\right|:\MM\text{ is a \ensuremath{k\times k} submatrix of \ensuremath{\AA}}\right\} .
\]

\begin{theorem}[Theorem 1 in~\cite{CGST1986}]
\label{thmCooketAl} Let $\bb\in\Q^{m}$ and $\cc\in\Q^{n}$. Let
$\xx^{*}$ be an optimal vertex of $\LP(\AA,\bb,\cc)$. If $\IP(\AA,\bb,\cc)$
is feasible, then there exists an optimal solution $\zz^{*}$ such
that\footnote{In fact, their argument yields an upper bound of $n\cdot\Delta_{n-1}\left(\AA\right)$.}
\[
\|\xx^{*}-\zz^{*}\|_{\infty}\le n\cdot\max\big\{\Delta_{k}(\AA):\ k=1,\ldots,n\big\}.
\]
\end{theorem}

Cook et al.'s result is truly a cornerstone result. 
Their proof technique has been used to establish proximity bounds involving other data parameters~\cite{L2019} and different norms~\cite{LPSX2020,LPSX2021}.
Furthermore, their result has been extended to derive proximity results for convex separable programs~\cite{G1990,H1990,W1991} (where the bound in Theorem~\ref{thmCooketAl} remains valid), for mixed integer programs~\cite{P2020}, and for random integer programs~\cite{OPW2020}.
%
%

Lov\'asz~\cite[Section 17.2]{AS1986} and Del Pia and Ma~\cite[Section 4]{DM2021}
identified tuples $(\AA,\bb,\cc)$ such that proximity is arbitrarily
close to the upper bound in Theorem~\ref{thmCooketAl}. However,
their examples crucially rely on the fact that $\bb$ can take arbitrary
rational values. In fact, Lov\'asz's example uses a totally unimodular
matrix $\AA$ while Del Pia and Ma use a unimodular matrix. 
Therefore, if the right hand sides $\bb$ in their examples were to be replaced by the integral rounded down vector $\lfloor\bb\rfloor$, then the polyhedron $\poly(\AA, \lfloor\bb\rfloor)$ would only have integral vertices. 

From an integer programming perspective, replacing $\bb$ with $\lfloor\bb\rfloor$
is natural as it strengthens the linear relaxation without cutting
off any feasible integer solutions. This leads to the proximity question
that we consider: \textit{For a given matrix $\AA$, find a value
$\pi\ge0$ such that} 
\[
\max_{\substack{\bb\in\Z^{m},\\
\xx^{*}\in\poly(\AA,\bb)\text{ is a vertex}\\
\poly(\AA,\bb)\cap\Z^{n}\neq\emptyset
}
}\min_{\zz^{*}\in\poly(\AA,\bb)\cap\Z^{n}}\|\xx^{*}-\zz^{*}\|_{\infty}\le\pi.
\]
Other than $\bb\in\Z^{m}$ versus $\bb\in\Q^{m}$, this question is equivalent to the question considered by Cook et al. 
The objective vector $\cc$ is not explicit in our formulation, but it is implicit when we consider
vertices of $\poly(\AA,\bb)$.

It was previously unknown if Cook et al.'s bound is tight when $\bb\in\Z^{m}$.
Actually, under this assumption there are many results indicating that proximity is independent of $n$:
Aliev et al.~\cite{AHO2019} prove that proximity is upper bounded by the largest entry of $\AA$ for knapsack polytopes, Veselov and Chirkov's result~\cite{VC2009} implies a proximity bound of $2$ when $\Delta_{n}(\AA)\le2$, and Aliev et al.~\cite{ACHW2021} prove a bound of $\Delta_{n}(\AA)$ for corner polyhedra.

Our main result is the first improvement on Cook et al.'s result.
Furthermore, our proof technique generalizes theirs, and we believe that it can be applied in the multiple settings where their technique is used. 
%

\begin{theorem}
\label{thm:main_thm_n_2}Let $n\geq2$, $\bb\in\Z^{m}$, and $\cc\in\Q^{n}$.
Let $\xx^{*}$ be an optimal vertex of $\LP(\AA,\bb,\cc)$. If $\IP(\AA,\bb,\cc)$
is feasible, then there exists an optimal solution $\zz^{*}$ such
that
\[
\|\xx^{*}-\zz^{*}\|_{\infty}<\frac{n}{2}\cdot\Delta_{n-1}(\AA).
\]
\end{theorem}

Our proof of Theorem~\ref{thm:main_thm_n_2} consists of three parts. 
First, we relate proximity to the volume of a certain polytope associated with the matrix $\AA$.
Second, we establish lower bounds on the volume of this polytope when $n=2$ and $n=3$; see Section~\ref{seclowdimension}. 
Third, we show how these proximity bounds in lower dimensions can be used to derive proximity bounds in higher dimensions; see Section~\ref{secLifting}.
We can improve Theorem~\ref{thm:main_thm_n_2} in certain settings.
%
For instance, if $n$ is a multiple of $3$ then we can replace the factor $\sfrac{n}{2}$ with $(\sfrac{\sqrt{2}}{3})\cdot n$; see Theorem~\ref{thm:prox_dim_1_2_3}.

We also consider the case when $\AA$ is \emph{strictly $\Delta$-modular}, that is, $\AA=\TT\BB$ for a totally unimodular matrix $\TT$ and a square integer matrix $\BB$ with determinant $\Delta$; see Section~\ref{secStrictDelta}.
Here we show the factor $\sfrac{n}{2}$ can be removed entirely, generalizing a recent result of N\"agele, Santiago, and Zenklusen \cite[Theorem 5]{NSZ2021}.
We also give essentially matching lower bounds; see Section~\ref{secLowerBound}.
It is more or less straightforward to find a 1-dimensional polytope $\poly(\AA, \bb) \subseteq \R^n$ with a matching lower bound on proximity when $n\geq2$.
Thus, our contribution with this lower bound is a \emph{full-dimensional} polytope $\poly\left(\AA,\bb\right)\subseteq\R^{n}$ with a unique integral point $\zz^{*}$, and a vertex $\xx^{*}$ \emph{sharing no common facet with} $\zz^{*}$, such that the proximity is, up to a constant additive factor, equal to $\Delta_{n-1}\left(\AA\right)$.
%
\begin{theorem}\label{thm:TU} Let $\Delta \ge 1$ and $n \ge 2$.
~
\begin{enumerate}
\item For all feasible instances $\IP\left(\AA,\bb,\cc\right)$ with $\AA$ strictly $\Delta$-modular and $\bb$ integral, and for all optimal vertices $\xx^{*}$ of $\LP(\AA,\bb,\cc)$, there exists an optimal solution $\zz^{*}$ of $\IP(\AA,\bb,\cc)$ such that 
\[
\|\xx^{*}-\zz^{*}\|_{\infty}\leq\max\left\{ \Delta_{n-1}(\AA),\Delta_{n}(\AA)\right\} -1.
\]
\item Let $\Delta \geq 3$. There exists a feasible instance $\IP\left(\AA,\bb,\cc\right)$ with $\AA$ strictly $\Delta$-modular and $\bb$ integral, and an optimal vertex $\xx^{*}$ of $\LP(\AA,\bb,\cc)$, such that every feasible integral solution $\zz^{*}$ of $\IP(\AA,\bb,\cc)$ satisfies
\[
\|\xx^{*}-\zz^{*}\|_{\infty} = \max\left\{ \Delta_{n-1}(\AA),\Delta_{n}(\AA)\right\} - 2.
\]
Moreover, $\poly\left(\AA,\bb\right)$ is full-dimensional,
and $\xx^{*}$ and $\zz^{*}$ do not lie on a common facet of $\poly\left(\AA,\bb\right)$.
\end{enumerate}
\end{theorem}

\subsection{Preliminaries and notation.}\label{ssecPrelim_not}

Here we outline the key objects and parameters used in the paper.

Let $\AA\in\Z^{m\times n}$ and $\bb\in\Z^{m}$ be such that $\poly(\AA, \bb) \cap \Z^n \neq \emptyset$.
For $I \subseteq [m] := \{1, \ldots, m\}$, we use $\AA_I$ and $\bb_I$ to denote the rows of $\AA$ and $\bb$ indexed by $I$.
If $I = \{i\}$, then we write $\aa_i^\top  := \AA_I$.
We use $\zero$ and $\one$ to denote the all zero and all one vector (in appropriate dimension).

We assume we are given an objective vector $\cc\in\Q^{n}$. 
We always assume that 
\begin{equation}\label{eqOnlyZero}
\poly\left(\AA,\bb\right)\cap\Z^{n}=\left\{ \zero\right\} .
\end{equation}
We may assume~\eqref{eqOnlyZero} without loss of generality for the following reason.
Suppose $\xx^*$ is an optimal vertex of $\LP(\AA, \bb, \cc)$, and let $I^* \subseteq\{1, \ldots, m\}$ be an optimal basis, i.e., $\xx^* = \AA_{I^*}^{-1} \bb_{I^*}^{\phantom{-1}}$. 
By perturbing $\cc$, we may assume $\cc^\top = \yy^{\t} \AA_{I^*}$ for a strictly positive vector $\yy \in \R^n$.
Let $\zz^*$ be an optimal solution $\IP(\AA, \bb, \cc)$ that is closest to $\xx^*$.
The polytope $\poly(\overline{\AA}, \overline{\bb}):= \{\xx \in \poly(\AA, \bb):\ \AA_{I^*} \xx \ge \AA_{I^*} \zz^*\}$ contains $\xx^*$ and $\zz^*$ and $\Delta_k(\overline{\AA}) = \Delta_k(\AA)$ for all $k \in [n]$.
Further, $\poly(\overline{\AA}, \overline{\bb}) \cap \Z^n = \{\zz^*\}$ because a different integer point $\ww^*$ in $\poly(\overline{\AA}, \overline{\bb})$ would satisfy $\cc^\top \ww^* = \yy^\top (\AA_{I^*}\ww^*) > \yy^\top (\AA_{I^*}\zz^*) = \cc^\top \zz^*$, contradicting the optimality of $\zz^*$. 
By replacing $\poly(\AA, \bb)$ by $\poly(\overline{\AA}, \overline{\bb})$ and translating $\zz^*$ to $\zero$ we may assume~\eqref{eqOnlyZero}.

Under Assumption~\ref{eqOnlyZero}, bounding proximity is equivalent to bounding 
\begin{equation}\label{eqProxWidth}
\max_{\xx \in \poly(\AA, \bb)}\ \|\xx\|_{\infty} = \max_{\aalpha \in \{\pm\ee_1, \ldots, \pm \ee_n\} }\ \max \left\{\aalpha^{\t} \xx:\  \xx \in \poly(\AA, \bb)\right\},
\end{equation}
where $\ee_1, \ldots, \ee_n \in \Z^n$ are the standard unit vectors.
In light of this, we analyze the maximum of an arbitrary linear form $\aalpha^\top \xx$ over $\poly(\AA, \bb)$ for $\aalpha \in \Z^n$.

We provide non-trivial bounds on the maximum of these linear forms for small values of $n$; see Section~\ref{seclowdimension}.
In order to lift low dimensional proximity results to higher dimensions (see Section~\ref{secLifting}), we consider slices of $\poly\left(\AA,\bb\right)$ through the origin induced by rows of $\AA$. 
Given $I\subseteq\left[m\right]$ with $|I| \le n-1$, define
\[
\poly_{I}\left(\AA,\bb\right):=\poly\left(\AA,\bb\right)\cap\ker\AA_{I}.
\]
We define $\poly_{\emptyset}(\AA, \bb) := \poly(\AA, \bb)$.
The bounds that we provide on $\aalpha^\top \xx$ are given in terms of the parameter
\[
\Delta_{I}\left(\AA,\objfn\right):=\frac{1}{\gcd\AA_{I}}\cdot\max\left\{ \left|\det\left(\begin{array}{c}
\objfn^{\t}\\
\AA_{K}
\end{array}\right)\right|:I\subseteq K\subseteq\left[m\right],\;\left|K\right|=n-1\right\},
\]
where
\[
\gcd\AA_{I} := \gcd \left\{\left|\det \MM \right|:\ \MM~\text{is a}~\text{rank}(\AA_I) \times \text{rank}(\AA_I)~\text{submatrix of}~\AA_I\right\}.
\]
In particular, we define $\prox_{I}\left(\AA,\bb,\objfn\right)$ to be the number satisfying
\begin{equation}\label{eqPseudoProx}
\max_{\xx\in\poly_{I}\left(\AA,\bb\right)}\objfn^{\t}\xx=\prox_{I}\left(\AA,\bb,\objfn\right)\Delta_{I}\left(\AA,\objfn\right).
\end{equation}
Maximizing over all $I\subseteq\left[m\right]$ such that $\poly_{I}\left(\AA,\bb\right)$ has a fixed dimension $d$, define
\[
\prox_{d}\left(\AA,\bb,\objfn\right):=\max_{I:\dim\poly_{I}\left(\AA,\bb\right)=d}\prox_{I}\left(\AA,\bb,\objfn\right).
\]
Equation~\eqref{eqPseudoProx} looks similar to the proximity bound we seek. 
However, $\Delta_{I}\left(\AA,\objfn\right)$ depends on $\objfn$, whereas our main result (Theorem~\ref{thm:main_thm_n_2}) only depends on $\Delta_{n-1}(\AA)$.
Later (see Remark~\ref{remarkAProof}), we will substitute $\ee_1, \ldots, \ee_n$ in for $\aalpha$ as in~\eqref{eqProxWidth}.
We also want to consider $I = \emptyset$ because $\poly_{\emptyset}(\AA, \bb)  = \poly(\AA, \bb)$ by definition. 
These substitutions will convert $\Delta_{I}\left(\AA,\objfn\right)$ to proving proximity of $\Delta_{n-1}(\AA)$ over $\poly(\AA, \bb)$ as desired in our proximity theorem.

Another important object for us is the following cone.
For $\xx^{*}\in\R^{n}$, define
\begin{align*}
\ccone\left(\AA,\xx^{*}\right):=\left\{\xx\in\R^{n}:
\begin{array}{rl}
\sign\left(\aa_{i}^{\t}\xx^{*}\right)\cdot\aa_{i}^{\t}\xx & \geq0~\forall~ i \in [m]~\text{such that}~\aa_{i}^{\t}\xx^{*}\neq0\\[.1 cm]
\aa_{i}^{\t}\xx & =0~\forall~ i \in [m]~\text{such that}~\aa_{i}^{\t}\xx^{*}=0
\end{array}\right\}.
\end{align*}
The cone $\ccone\left(\AA,\xx^{*}\right)$ serves as a key ingredient in the proof of Theorem~\ref{thmCooketAl} in \cite{CGST1986}. 
We also define the polytope
\[
\spindle\left(\AA,\xx^{*}\right):=\ccone\left(\AA,\xx^{*}\right)\cap\left(\xx^{*}-\ccone\left(\AA,\xx^{*}\right)\right).
\]
One checks that if $\xx^{*}\in\poly\left(\AA,\bb\right)$, then $\spindle\left(\AA,\xx^{*}\right)\subseteq\poly\left(\AA,\bb\right)$.
Moreover, if $\yy^{*}\in\spindle\left(\AA,\xx^{*}\right)$ then $\spindle\left(\AA,\yy^{*}\right)\subseteq\spindle\left(\AA,\xx^{*}\right)$.
Polytopes of this form, namely, ones in which every facet is incident to one of two distinguished vertices, known as \emph{spindles}, were used in \cite{S2012} to construct counterexamples to the Hirsch conjecture.

Throughout the paper, with the exception of Lemma~\ref{lem:lifting_lemma}, we fix $\AA\in\Z^{m\times n}$, $\bb\in\Z^{m}$, and $\objfn\in\Z^{n}$. 
Thus, in our notation we drop the dependence on $\AA$, $\bb$, and $\objfn$, and simply write $\poly,\poly_{I},\Delta_{I}$, and so on. 
We also set
\[
\prox:=\prox_{\dim \poly}\qquad\text{and}\qquad\Delta:=\Delta_{\emptyset}.
\]

\section{Proximity for 1, 2, and 3-dimensional polyhedra.}\label{seclowdimension}

In this section we improve upon the Cook et al. bound when $n=2$ and $n=3$:

\begin{theorem}
\label{thm:prox_dim_1_2_3}We have $\prox_{1},\prox_{2}<1$ and $\prox_{3}<\sqrt{2}$.
\end{theorem}

A useful fact for us is that we need only consider the case when the ambient dimension $n$ and the dimension $d$ of $\poly$ coincide.
\begin{lemma}
\label{lem:lifting_lemma}
Assume $I\subseteq\left[m\right]$ determines a linearly independent subset of the rows of $\AA$. 
Further assume the linear span of $\poly_{I}$ is $\ker\AA_{I}$, which has dimension $d$. 
Then there exists $\AAhat\in\Z^{\left(m-n+d\right)\times d}$, $\hat{\bb}\in\Z^{m-n+d}$, and $\hat{\objfn}\in\Z^{d}$ such that:
\begin{enumerate}
\item $\poly\bigl(\AAhat,\hat{\bb}\bigr)\cap\Z^{d}=\left\{ \zero\right\} $
and
\item $\prox_{I}\left(\AA,\bb,\objfn\right)=\prox\bigl(\AAhat,\hat{\bb},\hat{\objfn}\bigr)$.
\end{enumerate}
\end{lemma}

\begin{proof}
Without loss of generality, suppose $I = [n-d]$. 
Set $J := [n-d]$, $\bar{J} := \{n-d+1, \ldots, n\}$, and $\bar{I} := \{n-d+1, \ldots, m\}$.
Choose a unimodular matrix $\UU\in\Z^{n\times n}$ (e.g., via the Hermite Normal Form of $\AA_{[n]}$) such that 
\[
\AA\UU=\begin{pmatrix}\left(\AA\UU\right)_{I,J} & \zero\\
\left(\AA\UU\right)_{\bar{I},J} & \left(\AA\UU\right)_{\bar{I},\bar{J}}
\end{pmatrix}
\]
with $\left(\AA\UU\right)_{I,J}$ square and invertible. 

Set 
\(
\AAhat :=\left(\AA\UU\right)_{\bar{I},\bar{J}},
\)
\(
\hat{\bb} := \bb_{\bar{I}},
\)
and
\(
\hat{\objfn}^{\t} :=\left(\objfn^{\t}\UU\right)_{\bar{J}}.
\)
For $\xx \in \ker \AA_I$, we have 
\[
\zero = \AA_I \xx = \AA_I\UU \UU^{-1} \xx = [(\AA \UU)_{I,J} ~ \zero]\ \UU^{-1} \xx = (\AA\UU)_{I,J}(\UU^{-1}\xx)_J.
\]
Thus, $(\UU^{-1}\xx)_J = \zero$.
Hence, the map $\xx\mapsto\left(\UU^{-1}\xx\right)_{\bar{J}}$ is a linear isomorphism from $\ker\AA_{I}$ to $\R^{|\bar{J}|}=\R^{d}$, which restricts to a lattice isomorphism from $\ker\AA_{I}\cap\Z^{n}$ to $\Z^{d}$ and maps $\poly_{I}(\AA, \bb)$ to $\poly\bigl(\AAhat,\hat{\bb}\bigr)$.
It follows that $\poly\bigl(\AAhat,\hat{\bb}\bigr)\cap\Z^{d}=\left\{ \zero\right\} $.
For $\xx\in\ker\AA_{I}$, the equation $\left(\UU^{-1}\xx\right)_{J}=\zero$ implies that
\begin{equation}
\objfn^{\t}\xx=\objfn^{\t}\UU\UU^{-1}\xx=\hat{\objfn}^{\t}\left(\UU^{-1}\xx\right)_{\bar{J}}.\label{eq:numerator}
\end{equation}
Moreover, if $K\subseteq\bar{I}$ with $\left|K\right|=d-1$, then
\begin{align*}
\left|\det\left(\begin{array}{c}
\hat{\objfn}^{\t}\\
\AAhat_{K}
\end{array}\right)\right| 
& =\left|\det\left(\begin{array}{r@{\hskip 0 cm}r@{\hskip 0 cm}l}
\left(\right.&\objfn^{\t}\UU &\left.\right)_{\bar{J}}\\
\left(\right.&\AA\UU&\left.\right)_{K,\bar{J}}
\end{array}\right)\right|\\
 & =\frac{1}{\bigl|\det\left(\AA\UU\right)_{I,J}\bigr|}\cdot
 \left|\det\left(
 \begin{array}{r@{\hskip 0 cm}r@{\hskip 0 cm}lr@{\hskip 0 cm}r@{\hskip 0 cm}l}
 \left(\right.&\AA\UU&\left.\right)_{I,J} & &\zero~~&\\
\left(\right.&\objfn^{\t}\UU&\left.\right)_{J} & \left(\right.&\objfn^{\t}\UU&\left.\right)_{\bar{J}}\\
\left(\right.&\AA\UU&\left.\right)_{K,J} & \left(\right.&\AA\UU&\left.\right)_{K,\bar{J}}
\end{array}
\right)\right|\\
 & =\frac{1}{\gcd\AA_{I}}\cdot\left|\det\left(\begin{array}{c}
\objfn^{\t}\\
\AA_{I\cup K}
\end{array}\right)\right|,
\end{align*}
where we have used 
\(
\bigl|\det\left(\AA\UU\right)_{I,J}\bigr| = \gcd (\AA\UU)_I = \gcd \AA_I \UU = \gcd \AA_I.
\)
Taking the maximum over all such $K$, we get
\begin{equation}
\Delta\bigl(\AAhat,\hat{\objfn}\bigr)=\Delta_{I}\left(\AA,\objfn\right).\label{eq:denominator}
\end{equation}
Putting~\eqref{eq:numerator} and~\eqref{eq:denominator} together,
we get
\[
\prox\bigl(\AAhat,\hat{\bb},\hat{\objfn}\bigr)
=\max_{\yy\in\poly\left(\AAhat,\hat{\bb}\right)}\ \frac{\hat{\objfn}^{\t}\yy}{\Delta\bigl(\AAhat,\hat{\objfn}\bigr)}
=\max_{\xx\in\poly_{I}(\AA, \bb)}\ \frac{\objfn^{\t}\xx}{\Delta_{I}\left(\AA,\objfn\right)}
=\prox_{I}\left(\AA,\bb,\objfn\right).\qedhere
\]
\end{proof}

When it comes to bounding proximity, Lemma~\ref{lem:lifting_lemma} allows us to replace a not-necessarily full-dimensional instance with an equivalent full-dimensional instance in a lower-dimensional space. 
For example when the dimension of $\poly(\AA, \bb) \subseteq \R^n$ is $d=1$, we may assume $n=1$ and $\poly\left(\AA,\bb\right)$ is contained in the open interval $\left(-1,1\right)$, which immediately implies the first inequality $\prox_{1}<1$ of Theorem~\ref{thm:prox_dim_1_2_3}. 
The next lemma derives a general bound on $\prox_{i}$ in terms of the ambient dimension $n$ of the polyhedron, so it is useful to have this dimension reduction lemma. 

Define the polyhedron
\[
\poly_{\objfn}:=\left\{ \xx\in\R^{n}:\left|\AA\xx\right|\leq\one,\;\objfn^{\t}\xx=0\right\} .
\]
This is an $\left(n-1\right)$-dimensional polyhedron, which is bounded since the rows of $\AA$ positively span $\R^n$ by~\eqref{eqOnlyZero}.
We use $\vol_i(\cdot)$ to denote the $i$-dimensional Lebesgue measure.

\begin{lemma}
\label{lem:prox_volume}
Assume $n=d$. 
Then
\[
\prox_{n}<\frac{2^{n-1}\norm{\objfn}_{2}}{\vol_{n-1}\left(\poly_{\objfn}\right)\Delta}.
\]
\end{lemma}

\begin{proof}
Recall $\poly = \poly(\AA, \bb)$.
Let $\xx^{*}\in\poly$  attain the maximum of 
\[
\prox_{n}=\max_{\xx\in\poly}\frac{\objfn^{\t}\xx}{\Delta},
\]
which we assume is positive without loss of generality. 
Define the polytope
\[
\qoly\left(\xx^{*}\right):=\poly_{\objfn}+\left[-\xx^{*},\xx^{*}\right],
\]
which is $\zero$-symmetric and full-dimensional in $\R^{n}$. 
Observe that
\[
\vol_{n}\left(\qoly\left(\xx^{*}\right)\right)=\frac{2\prox_{n}\Delta}{\norm{\objfn}_{2}}\cdot\vol_{n-1}\left(\poly_{\objfn}\right).
\]
Let $\varepsilon,\delta>0$ be scalars chosen sufficiently small such that
\[
\vol_{n}\left(\left(1-\varepsilon\right)\qoly\left(\left(1+\delta\right)\xx^{*}\right)\right)=\vol_{n}\left(\qoly(\xx^*)\right)
\]
and
\[
\left(1-\varepsilon\right)\qoly\left(\left(1+\delta\right)\xx^{*}\right)\cap\Z^{n}\subseteq\qoly\left(\xx^{*}\right)\cap\Z^{n}.
\]

Assume to the contrary that $\vol_{n}\left(\qoly\left(\xx^{*}\right)\right)\geq2^{n}$.
By Minkowski's convex body theorem, there exists $\zz^{*} \in \Z^n \setminus \{\zero\}$ and $\lambda\in\left[0,1\right]$ such that $\zz^{*}$ can be uniquely written as
\(
\zz^{*}=\lambda\xx^{*}+\left(\zz^{*}-\lambda\xx^{*}\right)
\)
with $\zz^{*}-\lambda\xx^{*}\in\left( 1 - \epsilon \right)\poly_{\aalpha}$.
Hence,
\[
\left|\AA\left(\zz^{*}-\lambda\xx^{*}\right)\right|\leq\left(1-\varepsilon\right)\one.
\]
As $\poly\cap\Z^{n}=\left\{ \zero\right\} $ and $\zz^{*} \neq \zero$, there exists some row $\aa_{j}^{\t}$ of $\AA$ such that $\aa_{j}^{\t}\zz^{*}\geq\bb_{j}+1$. 
Since $\xx^{*}\in\poly(\AA, \bb)$, we also have $\aa_{j}^{\t}\xx^{*}\leq\bb_{j}$. 
Thus, we get
\begin{align*}
\bb_{j}+1\leq\aa_{j}^{\t}\zz^{*} & =\aa_{j}^{\t}\left(\lambda\xx^{*}\right)+\aa_{j}^{\t}\left(\zz^{*}-\lambda\xx^{*}\right)\leq\lambda\bb_{j}+\left(1-\varepsilon\right)<\bb_{j}+1.
\end{align*}
This is a contradiction.
Hence,
\[
\frac{2\prox_{n}\Delta}{\norm{\objfn}_{2}}\cdot\vol_{n-1}\left(\poly_{\objfn}\right)=\vol_{n}\left(\qoly\left(\xx^{*}\right)\right)<2^{n}.
\]
Rearranging yields the desired inequality.\qed
\end{proof}

\begin{remark}
Integrality of $\bb$, which is the key assumption of this paper, is used above in the assertion $\aa_{j}^{\t}\zz^{*}\geq\bb_{j}+1$.
If $\bb$ were not integral, then we would only be able to assert that $\aa_{j}^{\t}\zz^{*}\geq\lceil\bb_{j}\rceil$, which is not sufficient to complete the proof.
\end{remark}

\begin{proof}[of Theorem~\ref{thm:prox_dim_1_2_3} when $d = 2$]
By Lemma \ref{lem:lifting_lemma} we may assume $n=d=2$. 
The polytope $\poly_{\objfn}$ is an origin-symmetric line segment $\left[-\yy^{*},\yy^{*}\right]$, where $\yy^{*}\in\R^{2}$ satisfies $\objfn^{\t}\yy^{*}  =0$ and $\aa_{j}^{\t}\yy^{*}  =1$ for some $j\in\left[m\right]$. 
Hence
\[
\vol_{1}\left(\poly_{\objfn}\right)=2\norm{\yy^{*}}_{2}=\frac{2\norm{\objfn}_{2}}{\left|\det\left(\objfn\;\aa_{j}\right)\right|}.
\]
Applying Lemma \ref{lem:prox_volume}, we get
\[
\prox_{2}<\frac{2\norm{\objfn}_{2}}{\vol_{1}\left(\poly_{\objfn}\right)\Delta}=\frac{\left|\det\left(\objfn\;\aa_{j}\right)\right|}{\Delta}\leq1. \qedhere
\]
\end{proof}

We next find an upper bound for $\prox_{3}$. 
For this we apply the following classical result of Mahler~\cite{M1939B} (see also~\cite[Page 177]{G2007}) on the relationship between the area of a nonempty compact convex set $K\subseteq\R^{2}$ and the area of its polar
\[
K^{\circ}:=\left\{ \xx\in\R^{2}:\yy^{\t}\xx\leq1\text{ for all }\yy\in K\right\} .
\]

\begin{lemma}[Mahler's Inequality when $n = 2$~\cite{M1939B}]
Let $K\subseteq \R^{2}$ be nonempty compact and convex whose interior contains $\zero$. 
Then
\(
\vol_2\left(K\right)\vol_2\left(K^{\circ}\right)\geq8.
\)
\end{lemma}

\begin{proof}[of Theorem~\ref{thm:prox_dim_1_2_3} when $d = 3$]
By Lemma \ref{lem:lifting_lemma} we may assume $n=d=3$. 
Choose $I\subseteq\left[m\right]$ with $\left|I\right|=2$ such that
\[
\BB:=\begin{pmatrix}\objfn^{\t}\\
\AA_{I}
\end{pmatrix}
\]
satisfies $\left|\det\BB\right|=\Delta$. 
Let $\AA'$ denote the last two columns of $\AA\BB^{-1}$. 
Then
\begin{align*}
\BB\cdot\poly_{\objfn} & =\left\{ 0\right\} \times\qoly,
\end{align*}
where 
\[
\qoly:=\left\{ \xx\in\R^{2}:\left|\AA'\xx\right|\leq\one\right\} .
\]
We enumerate the rows of $\AA'$ as $\aa_{1}',\ldots,\aa_{m}'$. 
Since $\poly_{\aalpha}$ is a polytope, so is $\qoly$. 
The polar of $\qoly$ is the convex hull of the rows of $\AA'$:
\[
\qoly^{\circ}:=\conv\left\{ \aa_{i}':i\in\left[m\right]\right\} .
\]

Let $\tau:\R^2 \to \R^2$ denote the $90^{\circ}$ counterclockwise rotation in $\R^{2}$. 
Observe that
\(
\tau\left(\qoly^{\circ}\right)\subseteq\qoly.
\)
Indeed, for each pair $\{i,j\}\subseteq\left[m\right]$, we have
\[
\left|\tau\left(\aa_{i}'\right)^{\t}\aa_{j}'\right|=\left|\det\left(\aa_{i}'\;\aa_{j}'\right)\right|=\frac{\left|\det\left(\objfn\;\aa_{i}\;\aa_{j}\right)\right|}{\left|\det\BB\right|}\leq1.
\]
Hence, by Mahler's Inequality,
\[
\vol_{2}\left(\qoly\right)\geq\sqrt{\vol_{2}\left(\qoly\right)\vol_{2}\left(\tau\left(\qoly^{\circ}\right)\right)}=\sqrt{\vol_{2}\left(\qoly\right)\vol_{2}\left(\qoly^{\circ}\right)}\geq2\sqrt{2}.
\]
We have
\[
\vol_{2}\left(\qoly\right)=\frac{\left|\det\BB\right|}{\norm{\objfn}_{2}}\cdot\vol_{2}\left(\poly_{\objfn}\right).
\]
By Lemma~\ref{lem:prox_volume}, we get
\begin{align*}
\prox_{3}<\frac{4\norm{\objfn}_{2}}{\vol_{2}\left(\poly_{\objfn}\right)\Delta}\leq\frac{4}{2\sqrt{2}}=\sqrt{2}.\qedhere
\end{align*}
\end{proof}

\section{Lifting proximity results to higher dimensions.}\label{secLifting}

The next step is to show how proximity results for low dimensional polytopes can be used to derive proximity results for higher dimensional polytopes.

\begin{lemma}
\label{lem:spindle_down_to_zero}
Let $\yy^{*}\in\spindle\left(\xx^{*}\right)$, let $k:=\dim \spindle\left(\yy^{*}\right)$, and fix $d \in \{1, \ldots, k\}$.
There exists a $d$-face of $\spindle\left(\yy^{*}\right)$ incident to $\yy^{*}$ that intersects some $\left(k-d\right)$-face of $\spindle\left(\yy^{*}\right)$ incident to $\zero$.
\end{lemma}

\begin{proof}
Let $I\subseteq\left[m\right]$ index the components $i$ such that $\aa_{i}^{\t}\yy^{*}\neq0$. For $i\in I$ let $\hat{\aa}_{i}=\sign\left(\aa_{i}^{\t}\yy^{*}\right)\cdot\aa_{i}$.
The spindle $\spindle\left(\yy^{*}\right)$ can be written as
\begin{align*}
\spindle\left(\yy^{*}\right) = \left\{\xx \in \R^n:\ \zero\leq\hat{\aa}_{i}^{\t}\xx \leq\hat{\aa}_{i}^{\t}\yy^{*} ~\forall ~i \in I~\text{and}~\aa_{i}^{\t}\xx  =0 ~\forall ~ i \not \in I\right\}.
\end{align*}
The constraints are indexed by the disjoint union $I_{\zero}\cup I_{\yy^{*}}\cup\bar{I}$, where $I_{\zero}$ and $I_{\yy^{*}}$ denote the two copies of $I$ indexing constraints tight at $\zero$ and at $\yy^{*}$, respectively.
Let $J_{0},J_{1},\ldots,J_{r}$ be a sequence of feasible bases of this system, with corresponding basic feasible solutions $\zero=\yy^{(0)},\yy^{(1)},\ldots,\yy^{(r)}=\yy^{*}$ such that for each $i<r$, the symmetric difference of $J_{i+1}$ and $J_{i}$ is a 2-element subset of $I_{\zero}\cup I_{\yy^{*}}$.
We have $|J_{0}\cap I_{\yy^{*}}|=0$ and $|J_{r}\cap I_{\yy^{*}}|=k$, and $|J_{i+1}\backslash J_{i}|\leq1$ for each $i<r$. 
It follows that there must exist some $\ell$ such that $|J_{\ell}\cap I_{\yy^{*}}|=k-d$.
Since we always have $|J_{i}\cap(I_{\zero}\cup I_{\yy^{*}})|=k$ for every choice of $i$, we also get $|J_{\ell}\cap I_{\zero}|=d$. 

The basic feasible solution $\yy^{(\ell)}$ associated to $J_{\ell}$ is a vertex of the face of $\spindle\left(\yy^{*}\right)$ obtained by making the constraints of $J_{\ell}\cap I_{\yy^{*}}$ tight. 
It is also a vertex of the face of $\spindle\left(\yy^{*}\right)$ obtained by making the constraints of $J_{\ell}\cap I_{\zero}$ tight. 
These faces are contained in a $d$-face and a $\left(k-d\right)$-face, respectively.
\qed
\end{proof}

Lemma~\ref{lem:spindle_down_to_zero} will be used to create a path from one vertex of a spindle to another by traveling over $d$ dimensional faces.
In the next result, we apply proximity results to each $d$ dimensional face that we travel over.
This generalizes the proof of Cook et al., which can be interpreted as walking along edges of a spindle.

\begin{theorem}
\label{thm:proximity_template}Let $I\subseteq\left[m\right]$ index linearly independent rows of $\AA$ such that $\ker\AA_{I}$ is the linear span of $\poly$. 
Let $\dim\poly =: d = \sum_{i=0}^k d_i$ where each $d_{i}$ is a positive integer. 
Then
\[
\max_{\xx\in\poly}\ \objfn^{\t}\xx\leq\Delta_{I} \cdot \textstyle \sum_{i=0}^k \prox_{d_i}.
\]
\end{theorem}

\begin{remark}\label{remarkAProof}
Recall~\eqref{eqProxWidth}.
Our main result, Theorem~\ref{thm:main_thm_n_2}, is obtained by applying Theorem~\ref{thm:prox_dim_1_2_3}~and Theorem~\ref{thm:proximity_template} for each $\objfn \in \{\pm\ee_{1},\ldots, \pm\ee_{n}\}$ and $\left(d_{0},\ldots,d_{k}\right)=\left(3,\ldots,3,2,\ldots,2\right)$ with as few twos as possible such that $\sum_{i=0}^k d_{i}=d$.
\end{remark}

\begin{proof}[of Theorem~\ref{thm:proximity_template}]
Let $\xx^{*}$ maximize $\aalpha^{\t}\xx$ over $\poly$.
%
Build a sequence $\xx^{*}=:\xx_{0}^{*},\xx_{1}^{*},\ldots,\xx_{t}^{*}:=\zero$ of points inductively as follows. 
Assume $i\geq0$ and $\xx_{0}^{*},\ldots,\xx_{i}^{*}$ have been determined already.
If both 
\begin{equation}
i\leq k~\text{ and }~d_{i}<\dim\spindle\left(\xx_{i}^{*}\right),\label{eq:two_conditions}
\end{equation}
then we use Lemma~\ref{lem:spindle_down_to_zero} to choose a vertex $\xx_{i+1}^{*}$ of $\spindle\left(\xx_{i}^{*}\right)$ that is incident to both a $d_{i}$-dimensional face $F_{i}$ of $\spindle\left(\xx_{i}^{*}\right)$ containing $\xx_{i}^{*}$, as well as a $\left(\dim\spindle\left(\xx_{i}^{*}\right)-d_{i}\right)$-dimensional face $G_{i}$ of $\spindle\left(\xx_{i}^{*}\right)$ containing $\zero$.
Otherwise, if~\eqref{eq:two_conditions} fails, then we set $F_{i}=\spindle\left(\xx_{i}^{*}\right)$ and $\xx_{i+1}^{*}=\zero$, and we terminate the sequence by setting $t=i+1$.

Let $i \in \{0, \ldots, t-2\}$.
%
Note that $G_{i}$ is a proper face of $\spindle\left(\xx_{i}^{*}\right)$ containing both $\zero$ and $\xx_{i+1}^{*}$, and the only face of $\spindle\left(\xx_{t-1}^{*}\right)$ containing both $\zero$ and $\xx_{t-1}^{*}$ is $\spindle\left(\xx_{t-1}^{*}\right)$ itself.
One can see the latter claim by observing that the centre of symmetry of the centrally symmetric spindle $\spindle\left(\xx_{i}^{*}\right)$ is $\sfrac{1}{2} \cdot \xx_{i}^{*}$. 
Thus, $\xx_{i+1}^{*}\neq\zero$.
This shows
\begin{equation}\label{eqBigDim}
\dim\spindle\left(\xx_{i+1}^{*}\right)\geq1.
\end{equation}
Moreover, as both $G_{i}$ and $\spindle\left(\xx_{i+1}^{*}\right)$ are contained in the affine (equivalently, linear) span of $G_{i}$, we must have
\begin{equation}\label{eqStepThrough}
\dim\spindle\left(\xx_{i+1}^{*}\right)\leq\dim G_{i}=\dim\spindle\left(\xx_{i}^{*}\right)-d_{i}.
\end{equation}
Applying~\eqref{eqBigDim} and then~\eqref{eqStepThrough} sequentially with $s \in \{t-2,t-3,\ldots,0\}$, we have
\[
\textstyle
1\leq\dim\spindle\left(\xx_{t-1}^{*}\right) 
  \leq\dim\spindle\left(\xx_{0}^{*}\right)-\sum_{s=0}^{t-2}\ d_{s} 
  \leq d-\sum_{s=0}^{t-2}\ d_{s},
\]
which is to say
\(
d=\sum_{s=0}^{k} d_s > \sum_{s=0}^{t-2}d_s.
\)
It follows that $t-1\leq k$. 

Now let $i \in \{0, \ldots, t-1\}$.
We have that $\xx_{i}^{*}-F_{i}$ is a face of $\spindle\left(\xx_{i}^{*}\right)$ containing $\zero$.
Choose an index set $I_i$, where $I\subseteq I_{i}\subseteq\left[m\right]$, such that the rows of $\AA_{I_i}$ are linearly independent and $\ker\AA_{I_{i}}$ is the linear span of $\xx_{i}^{*}-F_{i}$. 
We have
\[
\objfn^{\t}\left(\xx_{i}^{*}-\xx_{i+1}^{*}\right)\leq\max_{\xx\in\xx_{i}^{*}-F_{i}}\objfn^{\t}\xx\leq\max_{\xx\in\poly_{I_{i}}}\objfn^{\t}\xx\leq\prox_{I_{i}}\Delta_{I_{i}}.
\]
If $i<t-1$, then since $F_i$ is a $d_i$-dimensional face, we have
\(
\prox_{I_{i}}\Delta_{I_{i}}\leq \prox_{d_{i}}\Delta_{I}
\)
for $i\in \{0,\ldots,t-2\}$.
Otherwise $i=t-1$, in which case one of the inequalities in~\eqref{eq:two_conditions} fails. We have established that $t-1\leq k$, thus
\begin{equation}
d_{t-1}\geq\dim\spindle\left(\xx_{t-1}^{*}\right)=\dim F_{t-1}.\label{eq:dim_last_face}
\end{equation}
and hence
\(
\prox_{I_{t-1}}\Delta_{I_{t-1}}\leq \prox_{d_{t-1}}\Delta_{I}
\).
Putting this all together we get
\[
\max_{\xx\in\poly}\ \objfn^{\t}\xx\ = \objfn^{\t}\xx^* =
\sum_{i=0}^{t-1}\ \objfn^{\t}\left(\xx_{i}^{*}-\xx_{i+1}^{*}\right)\ \leq\ \sum_{i=0}^{t-1}\ \prox_{I_{i}}\Delta_{I_{i}}\
\leq\ \Delta_I \cdot \sum_{i=0}^k\prox_{d_i} \qedhere
\]
\end{proof}

\section{Proximity in the strictly $\Delta$-modular case.}\label{secStrictDelta}

%
In this section we assume $\AA = \TT\BB$, where $\TT$ is totally unimodular and $\BB$ is an invertible square matrix with $|\det \BB| = \Delta = \Delta_n(\AA)$.
%

%

%
%
%
%
%
The following lemma is similar to \cite[Lemma 29 and Lemma 30]{NSZ2021} after linear transformation with $\BB$.
\begin{lemma}
\label{lem:strictlydeltamodrays} 
Set $\Lambda:=\BB^{-1}\Z^{n}$ and let $\xx^{*}\in\Lambda$.
Then each ray of $\ccone\left(\xx^{*}\right)$ contains a primitive vector in $\Lambda$, and $\xx^{*}$ can be written as a non-negative integral combination of those vectors. 
\end{lemma}

\begin{proof}
Let $I\subseteq\left[m\right]$ index a one-dimensional subspace $\ker\AA_{I}$
which contains a ray of $\ccone\left(\xx^{*}\right)$. 
Further, let $j\in\left[m\right]\backslash I $ index another row of $\AA$ such that $\AA_{I\cup j}$ is invertible with last row $\AA_{j}$.
We choose the following scaled vector 
\begin{align}\label{eq:latticepointray}
\rr:=\AA_{I\cup j}^{-1}\ee_{n}=\BB^{-1}\TT_{I\cup j}^{-1}\ee_{n}\in\ccone\left(\xx^{*}\right)\cap\ker\AA_{I}.
\end{align}
We have $\TT_{I\cup \{j\}}^{-1}\ee_{n}\in\Z^{n}$, so $\rr\in\Lambda$, because $\TT$ is totally unimodular.
The first claim follows, as the existence of a nonzero lattice vector on a ray implies the existence of a primitive lattice vector.

For the latter statement we study $\spindle\left(\xx^{*}\right)$.
If $\spindle\left(\xx^{*}\right)$ is zero-dimensional, then $\xx^{*}=\zero$ and we are done. 
If $\spindle\left(\xx^{*}\right)$ is one-dimensional, then $\xx^{*}$ is by construction an integer multiple of some primitive lattice ray. 
Hence, we assume that $\spindle\left(\xx^{*}\right)$ is at least two-dimensional. 
Choose a vertex $\vv$ adjacent to $\zero$ which is not $\xx^{*}$. 
As the constraint matrix defining $\spindle\left(\xx^{*}\right)$ is strictly $\Delta$-modular, every vertex of $\spindle\left(\xx^{*}\right)$ is in $\Lambda$, in particular $\vv\in\Lambda$. 
Thus, $\vv$ is an integer multiple of some primitive vector in $\Lambda$. 
Furthermore, the symmetry of $\spindle\left(\xx^{*}\right)$ implies 
\(
\xx^{*}-\vv\in\Lambda
\)
and is a vertex of $\spindle\left(\xx^*\right)$ adjancent to $\xx^*$. It follows there exists a constraint of $\spindle\left(\xx^*\right)$ tight at $\xx^*-\vv$ and $\zero$ but not $\xx^*$, and this implies that the dimension of $\spindle\left(\xx^*-\vv\right)$ is strictly smaller than $\spindle\left(\xx^*\right)$.  
We may therefore repeat the procedure with $\spindle\left(\xx^{*}-\vv\right)$ and so on, and termination is guaranteed when we reach the origin.
\qed
\end{proof}

\begin{proof}[of Theorem~\ref{thm:TU} Part 1]
%
%
Recall $\poly = \poly(\AA,\bb)$. Set $\Lambda:=\BB^{-1}\Z^{n}$, and choose a vertex $\xx^*$ of $\poly$ such that $\norm{\xx^*}_{\infty}$ is as large as possible.
Every vertex of $\poly$ is in $\Lambda$, so $\xx^{*}\in\Lambda$.
Lemma \ref{lem:strictlydeltamodrays} yields 
\begin{align}
\xx^{*}=\sum_{s=1}^t\lambda_{s}\rr_{s}\label{eq:positive_combination}
\end{align}
where $\rr_{1},\ldots,\rr_{t}\in\ccone\left(\xx^{*}\right)$ denote the
primitive vectors in $\Lambda$ and $\lambda_{1},\ldots,\lambda_{t}\in\Z_{\geq0}$.

Observe that each subsum of the right side in~\eqref{eq:positive_combination} is an element in $\spindle\left(\xx^{*}\right)$. 
Further, recall that $\spindle\left(\xx^{*}\right)\subseteq\poly$.
Let $N:=\lambda_{1}+\cdots+\lambda_{t}$. 
We choose a sequence $\zero=\xx^{(0)},\xx^{(1)},\ldots,\xx^{(N)}=\xx^{*}$ such that 
\begin{align*}
\xx^{(i)}-\xx^{(i-1)}\in\left\{ \rr_{1},\ldots,\rr_{t}\right\} 
\end{align*}
for all $i\in [N]$. 
Thus, $\xx^{(i)}\in\Lambda$ for $i\in\left[N\right]$ and all these vectors are pairwise distinct elements in $\spindle\left(\xx^{*}\right)$.
If $N\geq\Delta$, then by $\spindle\left(\xx^{*}\right)\cap\Z^{n}=\left\{ \zero\right\} $ and the pigeonhole principle there are $\xx^{(i)}$ and $\xx^{(j)}$ for $i<j$ that lie in the same residue class of $\Lambda$ modulo $\Z^{n}$. 
Hence, we have the contradiction
\begin{align*}
\zero\neq\xx^{(j)}-\xx^{(i)}\in\spindle\left(\xx^{*}\right)\cap\Z^{n}\subseteq\poly\cap\Z^{n}.
\end{align*}
We proceed with $N\leq\Delta-1$. 
We have 
\(
\norm{\rr_{s}}_{\infty}\leq \frac{\Delta_{n-1}(\AA)}{\Delta}
\)
for all $s\in\left[t\right]$ by Cramer's rule applied to (\ref{eq:latticepointray}). 
Altogether, this yields 
\begin{align*}
\norm{\xx^{*}}_{\infty}\leq\sum_{s=1}^{t}\lambda_{s}\norm{\rr_{s}}_{\infty}\leq \frac{\Delta - 1}{\Delta} \Delta_{n-1}(\AA) \leq \max\left\{ \Delta_{n-1}(\AA),\Delta_{n}(\AA)\right\}-1.
\end{align*}
\qed
\end{proof}

\section{A lower bound example.}\label{secLowerBound}

%
%
The following construction proves Theorem~\ref{thm:TU}, Part 2. Let $\Delta\geq 3$. 
Fix the matrix
\begin{align*}
\BB :=\left(\begin{array}{cc}
{\ve I}_{n-1} & \zero\\
{\ve{\beta}}^{k} & \Delta
\end{array}\right)\in\Z^{n\times n},
\end{align*}
where ${\ve I}_{n-1}$ denotes the $(n-1)\times(n-1)$ unit matrix and, for $0\leq k \leq n-1$,
\begin{align*}
{\ve{\beta}}^{k}:=(\underbrace{0,\ldots,0}_{k\;\text{zeros}},\Delta-1,\ldots,\Delta-1).
\end{align*}
%
%
%
%
%
As a first step, we define the parallelepiped 
\begin{align*}
	\poly(\BB) :=\left\{ \xx\in\R^{n}:\zero\leq\BB\xx\leq\left(\begin{array}{c}
		\one_{k}\\
		\Delta-n+k\\
		\one_{n-k-1}
	\end{array}\right)\right\}
\end{align*}
for $\Delta-n+k\geq1$. 
Using the fact that the first $k$ columns of $\BB^{-1}$ are integral, one can show $\bigl| \poly(\BB)\cap \Z^n \bigr| = 2^k$. 
In order to cut off all non-zero integer points in $\poly(\BB)$, we define for $k\geq 1$ the row vectors

\begin{align*}
\aa_{i}^{\t}:=(0,\ldots,\hspace{-.3 cm}\underbrace{1}_{i\text{-th column}}\hspace{-.3 cm},\ldots,0,\hspace{-.5 cm}\underbrace{-\Delta}_{(k+1)\text{-st column}}\hspace{-.5 cm},\ldots,-\Delta)
\end{align*}
for each $i\in [k]$. 
The resulting polytope is
\begin{align*}
\poly_{\Delta,n,k} :=\poly(\BB)\cap\left\{ \xx\in\R^{n}:\aa_{i}^{\t}\xx\leq0\text{ for all }i\in [k]\right\}.
\end{align*}
If $k=0$, then $\poly_{\Delta,n,0}=\poly(\BB)$. 
Let $\xx^{*}$ denote the only vertex in $\poly(\BB)$ that does not share a facet with $\zero$.
Further, let $\AA$ and $\bb$ be such that $\poly(\AA,\bb)=\poly_{\Delta,n,k}$.
Then one can show that

\begin{enumerate}
\item $\xx^{*}\in\poly_{\Delta,n,k}$ and $\xx^*$ does not share a facet with ${\ve 0}$,
\item $\poly_{\Delta,n,k}\cap\Z^{n}=\{\zero\}$, 
\item $\AA$ is strictly $\Delta$-modular.
\end{enumerate}
We select $\poly_{\Delta,n,n-2}$ and get
\begin{align*}
	\norm{\xx^* - {\ve 0}}_{\infty} = \left| \xx^*_{n - 1}\right| = \Delta - 2.
\end{align*}
Observe that $\Delta_{n-1}(\AA) = \Delta$ for $\poly_{\Delta,n,n-2}$ which proves Part 2 of Theorem~\ref{thm:TU}.

\begin{remark}
The polytope $\poly_{\Delta,n,n-1}$ does not work as an example since in this instance, the greatest common divisor of the $n$-th row of $\BB$ is $\Delta$. As a result, we only obtain the weak proximity bound $\norm{\xx^*}_{\infty}=1$.
However, a question related to the proximity question is to bound $\norm{\bb}_{\infty}$ given that $\poly\cap\Z^n=\{\zero\}$ and all constraints of $\poly$ are tight. The polytope $\poly_{\Delta,n,n-1}$ yields an example with $\norm{\bb}_{\infty} = \Delta - 1$.
\end{remark}
\newpage

\bibliographystyle{plain}
\bibliography{manuscript}

\end{document}